\pdfoutput=1
\RequirePackage{ifpdf}
\ifpdf 
\documentclass[pdftex]{sigma}
\else
\documentclass{sigma}
\fi

\numberwithin{equation}{section}

\newtheorem{Theorem}{Theorem}[section]
\newtheorem*{Theorem*}{Theorem}

\newtheorem{Lemma}[Theorem]{Lemma}

 { \theoremstyle{definition}

\newtheorem{Example}[Theorem]{Example}
\newtheorem{Remark}[Theorem]{Remark} }

\newcommand\id{\mathrm{id}}

\begin{document}

\newcommand{\arXivNumber}{2306.16411}

\renewcommand{\PaperNumber}{103}

\FirstPageHeading

\ShortArticleName{Expansions and Characterizations of Sieved Random Walk Polynomials}

\ArticleName{Expansions and Characterizations of Sieved Random\\ Walk Polynomials}

\Author{Stefan KAHLER~$^{\rm abc}$}

\AuthorNameForHeading{S.~Kahler}

\Address{$^{\rm a)}$~Fachgruppe Mathematik, RWTH Aachen University,\\
\hphantom{$^{\rm a)}$}~Pontdriesch 14-16, 52062 Aachen, Germany}
\EmailD{\href{kahler@mathematik.rwth-aachen.de}{kahler@mathematik.rwth-aachen.de}}
\Address{$^{\rm b)}$~Lehrstuhl A f\"{u}r Mathematik, RWTH Aachen University, 52056 Aachen, Germany}

\Address{$^{\rm c)}$~Department of Mathematics, Chair for Mathematical Modelling,\\
\hphantom{$^{\rm c)}$}~Chair for Mathematical Modeling of Biological Systems, Technical University of Munich,\\
\hphantom{$^{\rm c)}$}~Boltzmannstr.~3, 85747 Garching b.~M\"{u}nchen, Germany}

\ArticleDates{Received July 03, 2023, in final form December 01, 2023; Published online December 22, 2023}

\Abstract{We consider random walk polynomial sequences $(P_n(x))_{n\in\mathbb{N}_0}\subseteq\mathbb{R}[x]$ given by recurrence relations $P_0(x)=1$, $P_1(x)=x$, $x P_n(x)=(1-c_n)P_{n+1}(x)+c_n P_{n-1}(x),$ $n\in\mathbb{N}$ with $(c_n)_{n\in\mathbb{N}}\subseteq(0,1)$. For every $k\in\mathbb{N}$, the $k$-sieved polynomials $(P_n(x;k))_{n\in\mathbb{N}_0}$ arise from the recurrence coefficients $c(n;k):=c_{n/k}$ if $k|n$ and $c(n;k):=1/2$ otherwise. A main objective of this paper is to study expansions in the Chebyshev basis $\{T_n(x)\colon n\in\mathbb{N}_0\}$. As an application, we obtain explicit expansions for the sieved ultraspherical polynomials. Moreover, we introduce and study a sieved version $\mathrm{D}_k$ of the Askey--Wilson operator $\mathcal{D}_q$. It is motivated by the sieved ultraspherical polynomials, a generalization of the classical derivative and obtained from $\mathcal{D}_q$ by letting $q$ approach a $k$-th root of unity. However, for $k\geq2$ the new operator $\mathrm{D}_k$ on $\mathbb{R}[x]$ has an infinite-dimensional kernel (in contrast to its ancestor), which leads to additional degrees of freedom and characterization results for $k$-sieved random walk polynomials. Similar characterizations are obtained for a sieved averaging operator $\mathrm{A}_k$.}


\Classification{42C05; 33C47; 42C10}

\section{Introduction}\label{sec:intro}

In the theory of orthogonal polynomials, it is an important problem to identify properties which characterize specific classes. The literature is extensive; \cite{Al90} provides a valuable survey of such characterization results up to 1990. Our paper takes the newer contributions \cite{IO11,Ka16a,LO08} as a~starting point and deals with new characterizations of sieved random walk polynomials. Let~$\mu$ be a symmetric probability Borel measure on $\mathbb{R}$ with $|\mathsf{supp}\,\mu|=\infty$ and~${\mathsf{supp}\,\mu\subseteq[-1,1]}$, and let $(P_n(x))_{n\in\mathbb{N}_0}\subseteq\mathbb{R}[x]$ be the orthogonal polynomial sequence with respect to~$\mu$, normalized by~$P_n(1)=1$, $n\in\mathbb{N}_0$. In the following, we call such a measure $\mu$ just an `orthogonalization measure', and we call $(P_n(x))_{n\in\mathbb{N}_0}$ the (symmetric) `random walk polynomial sequence' (`RWPS') with respect to $\mu$. The relation to random walks is explained in \cite{CSvD98,vDS93}, for instance. Under the assumptions made on the support, it is well known that if an RWPS is orthogonal with respect to two orthogonalization measures, then these measures must coincide.\footnote{Standard results from the theory of orthogonal polynomials can be found in \cite{Ch78}, for instance.} Moreover, a~sequence~${(P_n(x))_{n\in\mathbb{N}_0}\subseteq\mathbb{R}[x]}$ is an RWPS if and only if it is given by a recurrence relation of the form $P_0(x)=1$ and
\begin{equation}\label{eq:recfund}
x P_n(x)=a_n P_{n+1}(x)+c_n P_{n-1}(x), \qquad n\in\mathbb{N}_0,
\end{equation}
where $c_0:=0$, $(c_n)_{n\in\mathbb{N}}\subseteq(0,1)$ and $a_n:=1-c_n$, $n\in\mathbb{N}_0$ \cite{Ch78,LO08}.\footnote{We make the convention that $0$ times something undefined shall be $0$.}

We consider two related sequences: on the one hand, if $(P_n(x))_{n\in\mathbb{N}_0}$ is an RWPS and ${k\in\mathbb{N}}$ is fixed, then let $(P_n(x;k))_{n\in\mathbb{N}_0}\subseteq\mathbb{R}[x]$ denote the `$k$-sieved RWPS' which corresponds to $(P_n(x))_{n\in\mathbb{N}_0}$, i.e., $(P_n(x;k))_{n\in\mathbb{N}_0}$ satisfies the recurrence relation $P_0(x;k)=1$,
\begin{equation*}
x P_n(x;k)=a(n;k)P_{n+1}(x;k)+c(n;k)P_{n-1}(x;k), \qquad n\in\mathbb{N}_0
\end{equation*}
with
\begin{equation*}
c(n;k):=\begin{cases} c_\frac{n}{k}, & k|n, \\ \frac{1}{2}, & \text{else}, \end{cases}
\end{equation*}
and $a(n;k):=1-c(n;k)$, $n\in\mathbb{N}_0$. Such sieved RWPS and related concepts have been studied in a series of papers by Ismail et al.\ (we particularly point out \cite{IL92}). We refer to the seminal paper~\cite{GvA88} of Geronimo and Van Assche which studies sieved polynomials via polynomial mappings (particularly to \cite[Section VI]{GvA88}). Sieved polynomials are a very fruitful topic in the theory of orthogonal polynomials; more recent contributions in this context are \cite{CdJP20,WLXZ16}, for instance.

On the other hand, for an RWPS $(P_n(x))_{n\in\mathbb{N}_0}$ let $(P_n^{\ast}(x))_{n\in\mathbb{N}_0}\subseteq\mathbb{R}[x]$ denote the polynomials which are orthogonal with respect to $\mathrm{d}\mu^{\ast}(x):=\big(1-x^2\big)\mathrm{d}\mu(x)$,
normalized by $P_n^{\ast}(1)=1$, $n\in\mathbb{N}_0$.\footnote{Note that $\mu^{\ast}$ is no longer a probability measure.} Let $h\colon \mathbb{N}_0\rightarrow(0,\infty)$ be given by~\cite{LO08}
\begin{equation*}
h(n):=\frac{1}{\int_\mathbb{R}\!P_n^2(x)\,\mathrm{d}\mu(x)}=\begin{cases} 1, & n=0, \\ \displaystyle \prod_{j=1}^n\frac{a_{j-1}}{c_j}, & \text{else}. \end{cases}
\end{equation*}
Then one explicitly has (cf.\ \cite{Ka16a})
\begin{equation}\label{eq:kernelstarprior}
P_n^{\ast}(x)=C_n^{\ast}\frac{P_{n+2}(x)-P_n(x)}{1-x^2}=
\frac{\sum_{k=0}^{\lfloor\frac{n}{2}\rfloor}h(n-2k)P_{n-2k}(x)}{\sum_{k=0}^{\lfloor\frac{n}{2}\rfloor}h(n-2k)}, \qquad n\in\mathbb{N}_0,
\end{equation}
where $C_n^{\ast}\in\mathbb{R}\backslash\{0\}$ depends on $n$ but is independent of $x$. The first equality in \eqref{eq:kernelstarprior} is an immediate consequence from the observation that
\begin{align*}
\int_\mathbb{R}\big(1-x^2\big)P_n^{\ast}(x)P_k(x)\mathrm{d}\mu(x)&=\int_\mathbb{R}P_n^{\ast}(x)P_k(x)\mathrm{d}\mu^{\ast}(x)=0,
\end{align*}
$n\in\mathbb{N}$, $k\in\{0,\ldots,n-1\}$, which yields that $\big(1-x^2\big)P_n^{\ast}(x)$ must be a linear combination of $P_n(x)$ and $P_{n+2}(x)$; since $\big(1-x^2\big)P_n^{\ast}(x)$ vanishes for $x=1$, the occurring linearization coefficients must be equal up to sign. The second equality in \eqref{eq:kernelstarprior} can be seen as follows: using \eqref{eq:recfund}, it is easy to see that (cf.\ \cite{Ka16b})
\begin{gather*}
\big(1-x^2\big)\sum_{k=0}^{\lfloor\frac{n}{2}\rfloor}h(n-2k)P_{n-2k}(x)=-c_{n+1}c_{n+2}h(n+2)[P_{n+2}(x)-P_n(x)],\qquad n\in\mathbb{N}_0.
\end{gather*}
Therefore, the second equality in \eqref{eq:kernelstarprior} follows from the first. Moreover, we see that
\begin{gather*}
C_n^{\ast}=-\frac{c_{n+1}c_{n+2}h(n+2)}{\sum_{k=0}^{\lfloor\frac{n}{2}\rfloor}h(n-2k)}.
\end{gather*}
Via the Christoffel--Darboux formula (cf.\ \cite{Ch78}), one can show that $C_n^{\ast}$ is also given by
\begin{equation*}
C_n^{\ast}=-\frac{2c_{n+1}c_{n+2}h(n+2)}{\sum_{k=0}^n h(k)+c_{n+1}h(n+1)}.
\end{equation*}
Recall that, given some $q\in(0,1)$, the Askey--Wilson operator $\mathcal{D}_q\colon \mathbb{R}[x]\rightarrow\mathbb{R}[x]$ is defined by linearity and the action~\cite{Is09,IO11}
\begin{equation}\label{eq:askeywilsondef}
\mathcal{D}_q T_n(x)=\frac{q^{\frac{n}{2}}-q^{-\frac{n}{2}}}{\sqrt{q}-\frac{1}{\sqrt{q}}}U_{n-1}(x),\qquad n\in\mathbb{N}_0,
\end{equation}
where $U_{-1}(x):=0$ and $(T_n(x))_{n\in\mathbb{N}_0}$, $(U_n(x))_{n\in\mathbb{N}_0}$ denote the sequences of Chebyshev polynomials of the first and second kind, so $T_n(\cos(\theta))=\cos(n\theta)$, $U_n(\cos(\theta))=\sin((n+1)\theta)/\sin(\theta)$, ${T_0(x)=U_0(x)=1}$, $T_1(x)=x$, $U_1(x)=2x$,
\begin{equation*}
x T_n(x)=\frac{1}{2}T_{n+1}(x)+\frac{1}{2}T_{n-1}(x),x U_n(x)=\frac{1}{2}U_{n+1}(x)+\frac{1}{2}U_{n-1}(x),\qquad n\in\mathbb{N}
\end{equation*}
and
\begin{equation}\label{eq:chebsecondfirstrelation}
U_n(x)=(n+1)T_n^{\ast}(x), \qquad n\in\mathbb{N}_0.
\end{equation}
Note that $(T_n(x))_{n\in\mathbb{N}_0}$ is the only RWPS which is invariant under sieving with arbitrary $k$. The classical derivative $\mathrm{d}/\mathrm{d}x$ is the limiting case $q\to1$ of $\mathcal{D}_q$; more precisely, \eqref{eq:askeywilsondef} is a $q$-extension of the well-known relation
\begin{equation}\label{eq:chebobserv2}
\frac{\mathrm{d}}{\mathrm{d}x}T_n(x)=n U_{n-1}(x), \qquad n\in\mathbb{N}_0.
\end{equation}
Relations~\eqref{eq:askeywilsondef}, \eqref{eq:chebsecondfirstrelation} and \eqref{eq:chebobserv2} can be interpreted in the following way: if $P_n(x)=T_n(x)$, $n\in\mathbb{N}_0$, then $P_n^\prime(x)=P_n^\prime(1)P_{n-1}^{\ast}(x)$ and $\mathcal{D}_q P_n(x)=\mathcal{D}_q P_n(1)P_{n-1}^{\ast}(x)$, $n\in\mathbb{N}$. The question which RWPS share the first of these properties has been answered in \cite[Lemma~1, Theorem~1]{LO08} and involves the ultraspherical polynomials:

\begin{Theorem}[Lasser--Obermaier 2008]\label{thm:lasserobermaierref}
The following are equivalent:
\begin{itemize}\itemsep=0pt
 \item[$(i)$] $P_n(x)=P_n^{(1/(2c_1)-3/2)}(x), \qquad n\in\mathbb{N}_0$,
  \item[$(ii)$] $P_n^\prime(x)=P_n^\prime(1)P_{n-1}^{\ast}(x),\qquad n\in\mathbb{N}$.
 \end{itemize}
\end{Theorem}

\cite[Theorem 5.2]{IO11} gives a $q$-analogue:

\begin{Theorem}[Ismail--Obermaier 2011]\label{thm:ismailobermaiercontref}
Let $q\in(0,1)$, $\beta\in(0,1/\sqrt{q})$ and $A=\sqrt{\beta}/2+1/(2\sqrt{\beta})$. Moreover, let $(Q_n(x))_{n\in\mathbb{N}_0}$ be orthogonal with respect to a symmetric probability Borel measure~$\mu$ on $\mathbb{R}$ with $|\mathsf{supp}\,\mu|=\infty$ and $\mathsf{supp}\,\mu\subseteq[-A,A]$; furthermore, let $(Q_n(x))_{n\in\mathbb{N}_0}$ be normalized by $Q_n(A)=1(n\in\mathbb{N}_0)$, and let $Q_2(0)=-\beta(1-q)/\big(1-\beta^2 q\big)$. Then the following are equivalent:
\begin{itemize}\itemsep=0pt
\item[$(i)$] $Q_n(x)=P_n(x;\beta|q),\qquad n\in\mathbb{N}_0$,
\item[$(ii)$] $(\mathcal{D}_q Q_n(x))_{n\in\mathbb{N}}$ is orthogonal with respect to $\big(A^2-x^2\big)\mathrm{d}\mu(x)$.
\end{itemize}
\end{Theorem}

In Theorems~\ref{thm:lasserobermaierref} and~\ref{thm:ismailobermaiercontref}, $\big(P_n^{(\alpha)}(x)\big)_{n\in\mathbb{N}_0}$ and $(P_n(x;\beta|q))_{n\in\mathbb{N}_0}$ denote the sequences of ultraspherical polynomials which correspond to $\alpha>-1$ and continuous $q$-ultraspherical (Rogers) polynomials which correspond to suitable $q$ and $\beta$, respectively, normalized such that\footnote{The special case $\beta=1$ is excluded in both the ``standard'' normalization of the continuous $q$-ultraspherical polynomials and the original formulation of the cited result \cite[Theorem 5.2]{IO11}. However, Theorem~\ref{thm:ismailobermaiercontref} remains valid with the definition $P_n(x;1|q):=T_n(x)$, $n\in\mathbb{N}_0$ (cf.\ \cite{Ka16a}).} \[P_n^{(\alpha)}(1)=P_n(\sqrt{\beta}/2+1/(2\sqrt{\beta});\beta|q)=1, \qquad n\in\mathbb{N}_0.\] Explicit formulas can be found in \cite{Is09,IO11,KLS10,LO08}. At this stage, we just recall that \smash{$\big(P_n^{(\alpha)}(x)\big)_{n\in\mathbb{N}_0}$} is orthogonal with respect to the probability measure \[\mathrm{d}\mu(x)=\Gamma(2\alpha+2)/\big(2^{2\alpha+1}\Gamma(\alpha+1)^2\big)\cdot\big(1-x^2\big)^\alpha\chi_{(-1,1)}(x)\mathrm{d}x\] and has the recurrence coefficients $c_n=n/(2n+2\alpha+1)$, $n\in\mathbb{N}$ \cite{LO08}. Furthermore, we note the striking limit relation
\begin{equation}\label{eq:limitmotivation}
\lim_{s\to1}P_n\left(x;s^{\alpha k+\frac{k}{2}}|s {\rm e}^\frac{2\pi {\rm i}}{k}\right)=P_n^{(\alpha)}(x;k), \qquad n\in\mathbb{N}_0
\end{equation}
between the continuous $q$-ultraspherical polynomials and the sieved ultraspherical polynomials \cite[Section 2]{AAA84}.

In \cite[Theorems 2.1 and 2.3]{Ka16a}, we sharpened the abovementioned Lasser--Obermaier result Theorem~\ref{thm:lasserobermaierref} and Ismail--Obermaier result Theorem~\ref{thm:ismailobermaiercontref} by showing that the characterizations remain valid if $n$ in (ii) is replaced by $2n-1$.\footnote{Results which are cited from \cite{Ka16a} can also be found in \cite{Ka16b}.}

A main purpose of the present paper is to study the interplay between the transitions $(P_n(x))_{n\in\mathbb{N}_0}\longrightarrow(P_n(x;k))_{n\in\mathbb{N}_0}$, $(P_n(x))_{n\in\mathbb{N}_0}\longrightarrow(P_n^{\ast}(x))_{n\in\mathbb{N}_0}$ and a ``sieved version'' of the Askey--Wilson operator. The limit relation \eqref{eq:limitmotivation} motivates our research in the following way: if one defines a corresponding ``sieved Askey--Wilson operator'' $\mathrm{D}_k\colon \mathbb{R}[x]\rightarrow\mathbb{R}[x]$ via
\begin{align}
\notag \mathrm{D}_k T_n(x):&=\lim_{s\to1}\frac{\left(\sqrt{s {\rm e}^\frac{2\pi {\rm i}}{k}}\right)^n-\left(\sqrt{s {\rm e}^\frac{2\pi {\rm i}}{k}}\right)^{-n}}{\sqrt{s {\rm e}^\frac{2\pi {\rm i}}{k}}-\frac{1}{\sqrt{s {\rm e}^\frac{2\pi {\rm i}}{k}}}}U_{n-1}(x)\\
\label{eq:Dkdef} &=U_{n-1}\left(\left|\cos\left(\frac{\pi}{k}\right)\right|\right)U_{n-1}(x), \qquad n\in\mathbb{N}_0
\end{align}
and linear extension, it is a natural question to ask whether, in analogy to Theorems~\ref{thm:lasserobermaierref} and~\ref{thm:ismailobermaiercontref}, $\mathrm{D}_k$ characterizes the sieved ultraspherical polynomials \smash{$\big(P_n^{(\alpha)}(x;k)\big)_{n\in\mathbb{N}_0}$}.\footnote{In the definition of $\mathrm{D}_k$ and in the definition of $\mathrm{A}_k$ below, $\sqrt{.}$ shall denote the principal value of the square root.} At first sight, this might be a reasonable conjecture. However, observe that if $k\geq2$, then the kernel of $\mathrm{D}_k$ becomes infinite-dimensional because $U_{n-1}(\cos(\pi/k))=\sin((n\pi)/k)/\sin(\pi/k)$ becomes zero for infinitely many $n\in\mathbb{N}_0$ (whereas the operators $\mathrm{d}/\mathrm{d}x=\mathrm{D}_1$ and $\mathcal{D}_q$ have finite-dimensional kernels). This important property might give reason to expect an additional degree of freedom. The situation is also very different to results of Ismail and Simeonov \cite{IP12} where Theorems~\ref{thm:lasserobermaierref} and~\ref{thm:ismailobermaiercontref} have been unified and extended to larger classes---but still for operators which reduce the polynomial degree by a fixed positive integer. In fact, the answer will depend on $k$. These results are given in Section~\ref{sec:sieved} and rely on an expansion result which is provided (and applied to an explicit example) in Section~\ref{sec:expansionssieved}. It turns out that as soon as $k\geq2$ the operator $\mathrm{D}_k$ does not lead to characterizations of sieved ultraspherical polynomials but to characterizations of \textit{arbitrary} $k$-sieved RWPS.\footnote{We say an RWPS $(P_n(x))_{n\in\mathbb{N}_0}$ to be `$k$-sieved' (without further specification) if it is $k$-sieved with respect to some RWPS or, equivalently, if $c_n=1/2$ if $k\,{\not|}\,n$.} Moreover, we present a characterization which involves the eigenvectors of the linear ``sieved averaging operator'' $\mathrm{A}_k\colon \mathbb{R}[x]\rightarrow\mathbb{R}[x]$,
\begin{align}
\notag \mathrm{A}_k T_n(x):&=\lim_{s\to1}\frac{\bigl(\sqrt{s {\rm e}^\frac{2\pi {\rm i}}{k}}\bigr)^n+\bigl(\sqrt{s {\rm e}^\frac{2\pi {\rm i}}{k}}\bigr)^{-n}}{2}T_n(x)\\
\label{eq:Akdef} &=T_n\left(\left|\cos\left(\frac{\pi}{k}\right)\right|\right)T_n(x), \qquad n\in\mathbb{N}_0.
\end{align}
The definition of $\mathrm{A}_k$ is motivated by \eqref{eq:limitmotivation} and the classical $q$-averaging operator $\mathcal{A}_q\colon\mathbb{R}[x]\rightarrow\mathbb{R}[x]$ \cite{Is09,IO11}
\begin{equation*}
\mathcal{A}_q T_n(x)=\frac{q^{\frac{n}{2}}+q^{-\frac{n}{2}}}{2}T_n(x), \qquad n\in\mathbb{N}_0.
\end{equation*}
$\mathcal{A}_q$ is a $q$-analogue of the identity operator and appears in the product rule of the Askey--Wilson operator. The characterization via $\mathrm{A}_k$ will also be given in Section~\ref{sec:sieved}, and it will be motivated by our following result on continuous $q$-ultraspherical polynomials \cite[Theorem~2.4]{Ka16a}:

\begin{Theorem}\label{thm:maincontqultrasph2}
Under the conditions of Theorem~{\rm \ref{thm:ismailobermaiercontref}} and the additional assumption that $\beta\leq1$, the following are equivalent: \begin{itemize}\itemsep=0pt
\item[$(i)$] $Q_n(x)=P_n(x;\beta|q)$, $n\in\mathbb{N}_0$,
\item[$(ii)$] the quotient \[
\frac{\int_\mathbb{R}\mathcal{A}_q Q_{n+1}(x)Q_{n-1}(x)\mathrm{d}\mu(x)}{\int_\mathbb{R}\mathcal{D}_q Q_{n+1}(x)Q_n(x)\mathrm{d}\mu(x)}\] is independent of $n\in\mathbb{N}$.
\end{itemize}
\end{Theorem}

Again it turns out that the passage from $\mathcal{A}_q$ to $\mathrm{A}_k$ leads to characterizations of \textit{arbitrary} $k$-sieved RWPS. However, the information contained in a previously specific---($q$-)ultraspherical---underlying structure is lost due to additional degrees of freedom. Here, these additional degrees of freedom can be traced back to the following fact (which is not obvious and will be established in Section~\ref{sec:sieved}, too): for any RWPS $(P_n(x))_{n\in\mathbb{N}_0}$, the integral \[\int_\mathbb{R}\mathcal{A}_k P_{n+1}(x)P_{n-1}(x)\mathrm{d}\mu(x)\] vanishes if $n\in\mathbb{N}$ is a multiple of $k$. Our characterization results with respect to $\mathrm{D}_k$ and $\mathrm{A}_k$ particularly prove a conjecture which we made in \cite{Ka16b}.

We remark that we used computer algebra systems (Maple) to find explicit formulas as in Example~\ref{example:sievedultrasph} below (which then can be verified by induction etc.), obtain factorizations of multivariate polynomials, get conjectures and so on. The final proofs can be understood without any computer usage, however.

\section{Expansions of sieved polynomials in the Chebyshev basis}\label{sec:expansionssieved}

In this section, we study expansions of sieved polynomials in the Chebyshev basis $\{T_n(x)\colon\allowbreak n\in\mathbb{N}_0\}$. Our result is suitable for explicit computations (see Example~\ref{example:sievedultrasph} below) and provides an important tool for the characterization results presented in Section~\ref{sec:sieved}.

Let $(P_n(x))_{n\in\mathbb{N}_0}$ be an RWPS as in Section~\ref{sec:intro}, and let $k\in\mathbb{N}$. We consider the connection coefficients to the Chebyshev polynomials of the first kind: for each $n\in\mathbb{N}_0$, we define a mapping $r_n\colon\{0,\ldots,\lfloor n/2\rfloor\}\rightarrow\mathbb{R}$ by the expansion
\begin{equation}\label{eq:connectionexpansion}
P_n(x)=\sum_{j=0}^{\lfloor\frac{n}{2}\rfloor}r_n(j)T_{n-2j}(x).
\end{equation}
Moreover, let the mappings $p_n,q_n\colon\{0,\ldots,\lfloor n/2\rfloor\}\rightarrow\mathbb{R}$, $n\in\mathbb{N}_0$, be recursively defined by
\begin{align*}
p_0(0)&:=0,\qquad
q_0(0):=1
\end{align*}
and the coupled system of recursions
\begin{align}
\label{eq:recnew1} p_n(j)&:=\begin{cases} 0, & n\quad \text{even and}\quad j=\frac{n}{2}, \\ \dfrac{(2a_n-1)q_{n-1}(j)+p_{n-1}(j-1)}{2a_n}, & \text{else}, \end{cases}\\
\label{eq:recnew2} q_n(j)&:=\begin{cases} p_{n-1}\left(\dfrac{n}{2}-1\right), & n\quad\text{even and}\quad j=\dfrac{n}{2}, \vspace{1mm}\\ \dfrac{q_{n-1}(j)+(2a_n-1)p_{n-1}(j-1)}{2a_n}, & \text{else} \end{cases}
\end{align}
for $n\in\mathbb{N}$ and $j\in\{0,\ldots,\left\lfloor n/2\right\rfloor\}$, where we set
\begin{equation}\label{eq:specialdefinition1}
p_{n-1}(-1):=0, \qquad n\in\mathbb{N}.
\end{equation}
The following theorem uses the sequences $(p_n)_{n\in\mathbb{N}_0}$ and $(q_n)_{n\in\mathbb{N}_0}$ to obtain the desired expansions of the sieved polynomials $(P_n(x;k))_{n\in\mathbb{N}_0}$ in the basis $\{T_n(x)\colon n\in\mathbb{N}_0\}$. Moreover, the theorem provides a possibility to compute $(p_n)_{n\in\mathbb{N}_0}$ and $(q_n)_{n\in\mathbb{N}_0}$ directly from $(r_n)_{n\in\mathbb{N}_0}$. To avoid case differentiations, we define
\begin{equation}\label{eq:specialdefinition2}
q_{2n-1}(n):=0, \qquad n\in\mathbb{N}.
\end{equation}

\begin{Theorem}\label{thm:mainsieved}
For every $k\in\mathbb{N}$, $n\in\mathbb{N}_0$ and $i\in\{0,\ldots,k-1\}$, one has
\begin{equation}\label{eq:mainsievedintheorem}
P_{k n+i}(x;k)=\sum_{j=0}^{\lfloor\frac{n}{2}\rfloor}[p_n(j)T_{k n-2j k-i}(x)+q_n(j)T_{k n-2j k+i}(x)].
\end{equation}
Moreover, one has
\begin{align}
\label{eq:rpqin} r_n(j)&=p_{n-1}(j-1)+q_{n-1}(j), \qquad n\in\mathbb{N},\quad j\in\left\{0,\ldots,\left\lfloor\frac{n}{2}\right\rfloor\right\},\\
\label{eq:rpqout} q_n(j)&=r_n(j)-p_n(j), \qquad n\in\mathbb{N}_0, \quad j\in\left\{0,\ldots,\left\lfloor\frac{n}{2}\right\rfloor\right\},\\
\label{eq:psum} p_n(j)&=\sum_{i=0}^j[r_n(i)-r_{n+1}(i)], \qquad n\in\mathbb{N}_0, \quad j\in\left\{0,\ldots,\left\lfloor\frac{n}{2}\right\rfloor\right\}.
\end{align}
\end{Theorem}

Concerning the expansions provided by Theorem~\ref{thm:mainsieved}, it is very remarkable that the coefficients of $T_{k n-2j k-i}(x)$ and $T_{k n-2j k+i}(x)$ do not rely on $i$, nor do they rely on $k$. Concerning well-definedness in \eqref{eq:mainsievedintheorem}, note that the ``polynomials'' $T_{-(k-1)}(x),\ldots,T_{-1}(x)$ are not defined; however, they only occur for even $n$ and together with a multiplication with ${p_n(\lfloor n/2\rfloor)=p_n(n/2)=\!0}$, and by our convention the product of $0$ and these undefined polynomials is interpreted as $0$. This convention will also be used in the following proof.

\begin{proof}[Proof of Theorem~\ref{thm:mainsieved}]
Let $k\geq2$ first. We establish the expansion \eqref{eq:mainsievedintheorem} via induction on $n\in\mathbb{N}_0$. It is clear from the recurrence relation for the Chebyshev polynomials of the first kind that $P_i(x;k)=T_i(x)$ for all $i\in\{0,\ldots,k\}$, so \eqref{eq:mainsievedintheorem} is true for $n=0$. Now let $n\in\mathbb{N}$ be arbitrary but fixed and assume the validity of \eqref{eq:mainsievedintheorem} for $n-1$. In particular, we then have
\begin{equation*}
P_{k n-2}(x;k)=\sum_{j=0}^{\lfloor\frac{n-1}{2}\rfloor}[p_{n-1}(j)T_{k n-2j k-2k+2}(x)+q_{n-1}(j)T_{k n-2j k-2}(x)].
\end{equation*}
Due to \eqref{eq:specialdefinition1} and \eqref{eq:specialdefinition2}, the latter equation can be rewritten as
\begin{equation}\label{eq:rec1a}
P_{k n-2}(x;k)=\sum_{j=0}^{\lfloor\frac{n}{2}\rfloor}[p_{n-1}(j-1)T_{k n-2j k+2}(x)+q_{n-1}(j)T_{k n-2j k-2}(x)].
\end{equation}
In the same way, we obtain
\begin{equation}\label{eq:rec1b}
P_{k n-1}(x;k)=\sum_{j=0}^{\lfloor\frac{n}{2}\rfloor}[p_{n-1}(j-1)T_{k n-2j k+1}(x)+q_{n-1}(j)T_{k n-2j k-1}(x)].
\end{equation}
We now use that $P_{k m}(x;k)=P_m(T_k(x))$, $m\in\mathbb{N}_0$ \cite[Theorem 1, Section VI]{GvA88}, which yields
\begin{equation}\label{eq:rec2}
P_{k n}(x;k)=\sum_{j=0}^{\lfloor\frac{n}{2}\rfloor}r_n(j)T_{n-2j}(T_k(x))=\sum_{j=0}^{\lfloor\frac{n}{2}\rfloor}r_n(j)T_{k n-2j k}(x).
\end{equation}
Combining \eqref{eq:rec1a}, \eqref{eq:rec1b} and \eqref{eq:rec2} with the relation
$2x P_{k n-1}(x;k)=P_{k n}(x;k)+P_{k n-2}(x;k)$ and using the recurrence relation for the Chebyshev polynomials of the first kind, we obtain that $r_n(j)=p_{n-1}(j-1)+q_{n-1}(j)$ for each $j\in\{0,\ldots,\left\lfloor n/2\right\rfloor\}$. Since
\begin{equation*}
p_{n-1}(j-1)+q_{n-1}(j)=p_n(j)+q_n(j)
\end{equation*}
for each $j\in\{0,\ldots,\left\lfloor n/2\right\rfloor\}$ (which makes use of the recursions \eqref{eq:recnew1} and \eqref{eq:recnew2}, as well as another use of definition \eqref{eq:specialdefinition2}), we therefore get
\begin{equation}\label{eq:rec2mod}
P_{k n}(x;k)=\sum_{j=0}^{\lfloor\frac{n}{2}\rfloor}[p_n(j)+q_n(j)]T_{k n-2j k}(x).
\end{equation}
We now combine \eqref{eq:rec1b} with \eqref{eq:rec2mod}, write
\begin{equation*}
2x T_{k n-2j k}(x)=T_{kn-2j k+1}(x)+T_{|k n-2j k-1|}(x), \qquad j\in\left\{0,\ldots,\left\lfloor\frac{n}{2}\right\rfloor\right\},
\end{equation*}
use \eqref{eq:recnew1}, \eqref{eq:recnew2} and definition \eqref{eq:specialdefinition2} again and obtain
\begin{align*}
2a_n P_{k n+1}(x;k)&=2x P_{k n}(x;k)-2c_n P_{k n-1}(x;k)\\
&=\sum_{j=0}^{\lfloor\frac{n}{2}\rfloor}[p_n(j)+q_n(j)-2c_n p_{n-1}(j-1)]T_{k n-2j k+1}(x)\\
&\quad+\sum_{j=0}^{\lfloor\frac{n}{2}\rfloor}[p_n(j)+q_n(j)-2c_n q_{n-1}(j)]T_{|k n-2j k-1|}(x)\\
&=2a_n\sum_{j=0}^{\lfloor\frac{n}{2}\rfloor}[p_n(j)T_{k n-2j k-1}(x)+q_n(j)T_{k n-2j k+1}(x)].
\end{align*}
Thus if $k=2$, then \eqref{eq:mainsievedintheorem} is shown. If $k\geq3$, we use induction on $i$ to prove that
\begin{equation}\label{eq:conrec1ind}
P_{k n+i}(x;k)=\sum_{j=0}^{\lfloor\frac{n}{2}\rfloor}[p_n(j)T_{k n-2j k-i}(x)+q_n(j)T_{k n-2j k+i}(x)], \qquad i\in\{0,\ldots,k-1\}
\end{equation}
and have already shown the initial step $i\in\{0,1\}$; we hence assume $i\in\{0,\ldots,k-3\}$ to be arbitrary but fixed and \eqref{eq:conrec1ind} to hold for $i,i+1$, and then calculate
\begin{align*}
P_{k n+i+2}(x;k)&=2x P_{k n+i+1}(x;k)-P_{k n+i}(x;k)\\
&=\sum_{j=0}^{\lfloor\frac{n}{2}\rfloor}[p_n(j)T_{k n-2j k-i-2}(x)+q_n(j)T_{k n-2j k+i+2}(x)].
\end{align*}
This finishes the proof of \eqref{eq:mainsievedintheorem} for $k\geq2$, and we have simultaneously established \eqref{eq:rpqin} and~\eqref{eq:rpqout}. \eqref{eq:mainsievedintheorem} for the remaining case $k=1$ is an immediate consequence of \eqref{eq:rpqout}. Finally, \eqref{eq:psum} can be seen as follows: let $n\in\mathbb{N}_0$ and $j\in\{0,\ldots,\left\lfloor n/2\right\rfloor\}$. By \eqref{eq:rpqin} and \eqref{eq:rpqout}, we have
\begin{equation*}
r_n(i)-r_{n+1}(i)=p_n(i)-p_n(i-1)
\end{equation*}
for all $i\in\{0,\ldots,j\}$. Taking the sum from $0$ to $j$ and using definition \eqref{eq:specialdefinition1}, we get
\begin{equation*}
\sum_{i=0}^j[r_n(i)-r_{n+1}(i)]=\sum_{i=0}^j[p_n(i)-p_n(i-1)]=p_n(j)
\end{equation*}
as desired.
\end{proof}

We now apply Theorem~\ref{thm:mainsieved} to the ultraspherical polynomials and obtain explicit expansions of the sieved ultraspherical polynomials with respect to the Chebyshev basis $\{T_n(x)\colon n\in\mathbb{N}_0\}$:

\begin{Example}[sieved ultraspherical polynomials]\label{example:sievedultrasph}
Let $P_n(x)=P_n^{(\alpha)}(x)$, $n\in\mathbb{N}_0$, be the sequence of ultraspherical polynomials which corresponds to $\alpha>-1$. The case $\alpha=-1/2$ corresponds to the Chebyshev polynomials of the first kind $(T_n(x))_{n\in\mathbb{N}_0}$ and is therefore trivial, so let $\alpha\neq-1/2$ from now on. Then $(r_n)_{n\in\mathbb{N}_0}$ is given by \cite[Theorem 9.1.1]{Is09}
\begin{equation}\label{eq:rsievedultrasph}
r_n(j)=\begin{cases} \displaystyle\binom{n}{\frac{n}{2}}\frac{\left(\alpha+\frac{1}{2}\right)_{\frac{n}{2}}^2}{(2\alpha+1)_n}, & n\quad\text{even and}\quad j=\frac{n}{2}, \\ \displaystyle  2\binom{n}{j}\frac{\left(\alpha+\frac{1}{2}\right)_j\left(\alpha+\frac{1}{2}\right)_{n-j}}{(2\alpha+1)_n}, & \text{else}. \end{cases}
\end{equation}
Applying \eqref{eq:psum}, via induction on $j$ we see that the relation between $(p_n)_{n\in\mathbb{N}_0}$ and $(r_n)_{n\in\mathbb{N}_0}$ becomes especially easy and reads
\begin{equation}\label{eq:psievedultrasph}
p_n(j)=\begin{cases} 0, & n\quad \text{even and}\quad j=\frac{n}{2}, \\ \dfrac{2j+2\alpha+1}{2n+4\alpha+2}r_n(j), & \text{else}. \end{cases}
\end{equation}
Theorem~\ref{thm:mainsieved}, \eqref{eq:rsievedultrasph} and \eqref{eq:psievedultrasph} allow an explicit computation of the sieved ultraspherical polynomials \smash{$\big(P_n^{(\alpha)}(x;k)\big)_{n\in\mathbb{N}_0}$}.
\end{Example}

\section{Characterizations via the sieved operators}\label{sec:sieved}

Let $(P_n(x))_{n\in\mathbb{N}_0}$ be an RWPS as in Section~\ref{sec:intro}. Moreover, let $k\in\mathbb{N}$ again. Following \cite{IO11,Ka16a,LO08} ($q$- and non-sieved analogues), we consider the Fourier coefficients which are associated with $\mathrm{D}_k$ \eqref{eq:Dkdef} and $\mathrm{A}_k$ \eqref{eq:Akdef}: for each $n\in\mathbb{N}_0$, we define mappings $\kappa_n(\cdot;k),\alpha_n(\cdot;k)\colon \mathbb{N}_0\rightarrow\mathbb{R}$ by the projections
\begin{align}
\label{eq:kappandeffirst} \kappa_n(j;k)&:=\int_\mathbb{R}\mathrm{D}_k P_n(x)P_j(x)\mathrm{d}\mu(x),\\
\label{eq:alphandeffirst} \alpha_n(j;k)&:=\int_\mathbb{R}\mathrm{A}_k P_n(x)P_j(x)\mathrm{d}\mu(x).
\end{align}
In other words, $\kappa_n(\cdot;k)$ and $\alpha_n(\cdot ;k)$ correspond to the expansions
\begin{align}
\label{eq:kappandefsecond} \mathrm{D}_k P_n(x)&=\sum_{j=0}^{n-1}\kappa_n(j;k)P_j(x)h(j),\kappa_n(j;k)=0, \qquad j\geq n,\\
\label{eq:alphandefsecond} \mathrm{A}_k P_n(x)&=\sum_{j=0}^n\alpha_n(j;k)P_j(x)h(j),\alpha_n(j;k)=0, \qquad j\geq n+1.
\end{align}
Due to the symmetry of $(P_n(x))_{n\in\mathbb{N}_0}$, we have $\kappa_n(j;k)=0$ if $n-j$ is even, and we have ${\alpha_n(j;k)=0}$ if $n-j$ is odd. Furthermore, note that the functions $\kappa_0(\cdot;k),\ldots,\kappa_{n+1}(\cdot;k)$, $\alpha_0(\cdot;k),\allowbreak\ldots,\alpha_n(\cdot;k)$ and $\alpha_{n+1}(\cdot;k)|_{\{0,\ldots,n\}}$ are uniquely determined by the recurrence coefficients $c_1,\allowbreak\ldots,c_n$. For brevity, we define $\sigma(\cdot;k)\colon\mathbb{N}\rightarrow\mathbb{R}$,
\begin{equation}\label{eq:sigmadef}
\sigma(n;k):=\kappa_n(n-1;k).
\end{equation}
We come to two further main results and give characterizations in terms of the sieved operators~$\mathrm{A}_k$ and $\mathrm{D}_k$. Theorem~\ref{thm:mainsievedaveraging} is the sieved analogue to Theorem~\ref{thm:maincontqultrasph2}. Theorem~\ref{thm:mainsievedaskeywilson} is the sieved analogue to the Lasser--Obermaier result Theorem~\ref{thm:lasserobermaierref} and Ismail--Obermaier result Theorem~\ref{thm:ismailobermaiercontref}. As soon as $k\geq2$ (and also for $k=1$ in Theorem~\ref{thm:mainsievedaveraging}), we do not obtain characterizations of sieved ultraspherical polynomials (as one might expect due to comparison to the cited theorems) but characterizations of arbitrary $k$-sieved RWPS.

\begin{Theorem}\label{thm:mainsievedaveraging}
If $k\in\mathbb{N}$, then the following are equivalent:
\begin{enumerate}\itemsep=0pt
\item[$(i)$] $(P_n(x))_{n\in\mathbb{N}_0}$ is $k$-sieved,
\item[$(ii)$] for each $n\in\mathbb{N}_0$, $P_n(x)$ is an eigenvector of $\mathrm{A}_k$,
\item[$(iii)$] $\alpha_{n+1}(n-1;k)=0$, $n\in\mathbb{N}$.
\end{enumerate}
If these equivalent conditions are satisfied, then $P_n(|\cos(\pi/k)|)=T_n(|\cos(\pi/k)|)$ is the eigenvalue of $\mathrm{A}_k$ which corresponds to the eigenvector $P_n(x)$, $n\in\mathbb{N}_0$.
\end{Theorem}

\begin{Theorem}\label{thm:mainsievedaskeywilson}
If $k\geq2$, then the following are equivalent:
\begin{enumerate}\itemsep=0pt
\item[$(i)$] $(P_n(x))_{n\in\mathbb{N}_0}$ is $k$-sieved,
\item[$(ii)$] $\mathrm{D}_k P_n(x)=\mathrm{D}_k P_n(1)P_{n-1}^{\ast}(x)$, $n\in\mathbb{N}$,
\item[$(iii)$] $\big(1-x^2\big)\mathrm{D}_k P_n(x)$ is orthogonal to $P_0(x),\ldots,P_{n-2}(x)$, $n\geq2$,
\item[$(iv)$] one has
\begin{equation*}
\kappa_{n+2}(n-1;k)=\sigma(n+2;k), \qquad n\in\mathbb{N},
\end{equation*}
and for every $n\in\mathbb{N}$ there is an $m\in\{0,\ldots,\left\lfloor(n-1)/2\right\rfloor\}$ such that
\begin{equation*}
\kappa_{n+4}(n-1-2m;k)=\sigma(n+4;k).
\end{equation*}
\end{enumerate}
If $k=1$, then $(ii)$, $(iii)$ and $(iv)$ are equivalent to
\begin{enumerate}\itemsep=0pt
\item[$(i')$] $P_n(x)=P_n^{(1/(2c_1)-3/2)}(x)$, $n\in\mathbb{N}_0$.
\end{enumerate}
\end{Theorem}

The characterization provided by (iii) of the previous theorem has the advantage that it is ``stable'' with respect to renormalization of the sequence $(P_n(x))_{n\in\mathbb{N}_0}$. The characterization provided by (iv) is the strongest one, however, because the functions $\kappa_n(\cdot;k)$ \eqref{eq:kappandeffirst} \eqref{eq:kappandefsecond} have to be considered just at some carefully chosen points.

Note that the formal limits ``$\mathrm{D}_\infty$'' and ``$\mathrm{A}_\infty$'' are included in our investigations because they coincide with $\mathrm{D}_1=\mathrm{d}/\mathrm{d}x$ and $\mathrm{A}_1=\id$.

Before coming to the proofs, we study some basic properties of $\mathrm{D}_k$ and $\mathrm{A}_k$. We will make use of the following well-known identities \cite[formulas~(22.7.25)--(22.7.28)]{AS65}:
\begin{align}
\label{eq:Chebrecurrencesim} \big(2x^2-2\big)U_{n-1}(x)&=T_{n+1}(x)-T_{n-1}(x), \qquad n\in\mathbb{N},\\
\label{eq:mixedsecond} 2T_m(x)U_{n-1}(x)&=U_{m+n-1}(x)+U_{n-m-1}(x), \qquad m,n\in\mathbb{N}_0,\quad m\leq n,\\
\label{eq:mixedthird} 2T_m(x)U_{n-1}(x)&=U_{m+n-1}(x)-U_{m-n-1}(x), \qquad m,n\in\mathbb{N}_0,\quad m\geq n.
\end{align}
The following lemma deals with the function $\sigma(.;k)$ \eqref{eq:sigmadef} and with special values of the functions $\alpha_n(.;k)$ \eqref{eq:alphandeffirst} \eqref{eq:alphandefsecond}. The analogous $q$- and non-sieved versions can be found in \cite{IO11,LO08} with similar proofs.

\begin{Lemma}\label{lma:last}
One has
\begin{enumerate}\itemsep=0pt
\item[$(i)$] $\alpha_n(n;k)=T_n(|\cos(\pi/k)|)/h(n)$, $n\in\mathbb{N}_0$,
\item[$(ii)$] the equation
\begin{equation*}
\alpha_n(n-2;k)=\frac{U_{n-2}\left(\left|\cos\left(\frac{\pi}{k}\right)\right|\right)\sin^2\left(\frac{\pi}{k}\right)\big(n-4\sum_{j=1}^{n-1}a_{j-1}c_j\big)}{2c_{n-1}c_n h(n)}
\end{equation*}
holds for all $n\geq2$,
\item[$(iii)$] $\sigma(n;k)=U_{n-1}(|\cos(\pi/k)|)/(c_n h(n))$, $n\in\mathbb{N}$.
\end{enumerate}
\end{Lemma}

\begin{proof}
(i) If one expands $P_n(x)$ as in \eqref{eq:connectionexpansion}, this is obvious from the definitions (in particular, use \eqref{eq:Akdef}).

(ii) Using \eqref{eq:Akdef}, \eqref{eq:connectionexpansion} and (i), we have
\begin{align*}
&\sum_{j=1}^{\lfloor\frac{n}{2}\rfloor}r_n(j)T_{n-2j}\left(\left|\cos\left(\frac{\pi}{k}\right)\right|\right)T_{n-2j}(x)\\
&\quad=\mathrm{A}_k P_n(x)-r_n(0)T_n\left(\left|\cos\left(\frac{\pi}{k}\right)\right|\right)T_n(x)\\
&\quad=\left[T_n\left(\left|\cos\left(\frac{\pi}{k}\right)\right|\right)r_n(1)+\alpha_n(n-2;k)h(n-2)r_{n-2}(0)\right]T_{n-2}(x)+R(x)
\end{align*}
for some $R(x)\in\mathbb{R}[x]$ with $\deg R(x)\leq n-3$; thus a comparison of the coefficients of~$T_{n-2}(x)$ and \eqref{eq:Chebrecurrencesim} yield
\begin{align*}
\alpha_n(n-2;k)h(n-2)r_{n-2}(0)&=\left[T_{n-2}\left(\left|\cos\left(\frac{\pi}{k}\right)\right|\right)-T_n\left(\left|\cos\left(\frac{\pi}{k}\right)\right|\right)\right]r_n(1)\\
&=2U_{n-2}\left(\left|\cos\left(\frac{\pi}{k}\right)\right|\right)\sin^2\left(\frac{\pi}{k}\right)r_n(1).
\end{align*}
Finally, since
\begin{equation*}
4a_{n-2}a_{n-1}r_n(1)=\Bigg[n-4\sum_{j=1}^{n-1}a_{j-1}c_j\Bigg]r_{n-2}(0)
\end{equation*}
(see the proof of \cite[Lemma 5.1]{IO11}) and $a_{n-2}a_{n-1}h(n-2)=c_{n-1}c_n h(n)$, we obtain the assertion.

(iii) While on the one hand one has
\begin{equation*}
\mathrm{D}_k P_n(x)=\sum_{j=0}^{\lfloor\frac{n}{2}\rfloor}r_n(j)U_{n-2j-1}\left(\left|\cos\left(\frac{\pi}{k}\right)\right|\right)U_{n-2j-1}(x)
\end{equation*}
by \eqref{eq:Dkdef}, on the other hand one has
\begin{align*}
\mathrm{D}_k P_n(x)&=\kappa_n(n-1;k)h(n-1)r_{n-1}(0)T_{n-1}(x)+R(x)\\
&=\frac{\kappa_n(n-1;k)h(n-1)r_{n-1}(0)}{2-\delta_{n,1}}U_{n-1}(x)+S(x)
\end{align*}
with polynomials $R(x),S(x)\in\mathbb{R}[x]$ with $\deg R(x),\deg S(x)\leq n-2$. Consequently, we have
\begin{equation*}
r_n(0)U_{n-1}\left(\left|\cos\left(\frac{\pi}{k}\right)\right|\right)=\frac{\kappa_n(n-1;k)h(n-1)r_{n-1}(0)}{2-\delta_{n,1}},
\end{equation*}
and as obviously $r_{n-1}(0)=(2-\delta_{n,1})a_{n-1}r_n(0)$ and $a_{n-1}h(n-1)=c_n h(n)$, the proof is complete (note that, by definition, $\kappa_n(n-1;k)=\sigma(n;k)$).
\end{proof}

We also investigate the product rule for the sieved Askey--Wilson operator $\mathrm{D}_k$. Its analogue for $\mathcal{D}_q$ has the same structure (see \cite{Is09,IO11}).

\begin{Lemma}\label{lma:productrule}
One has
\begin{equation*}
\mathrm{D}_k[P(x)Q(x)]=\mathrm{D}_k P(x)\mathrm{A}_k Q(x)+\mathrm{A}_k P(x)\mathrm{D}_k Q(x), \qquad P(x),Q(x)\in\mathbb{R}[x].
\end{equation*}
Consequently,
\begin{align}\label{eq:recurrencekappasieved}
&a_n\kappa_{n+1}(j;k)+c_n\kappa_{n-1}(j;k)\nonumber\\
&\quad=\left|\cos\left(\frac{\pi}{k}\right)\right|[a_j\kappa_n(j+1;k)+c_j\kappa_n(j-1;k)]+\alpha_n(j;k), \qquad n,j\in\mathbb{N}_0.
\end{align}
\end{Lemma}

\begin{proof}
Due to linearity, it clearly suffices to establish that
\begin{equation*}
\mathrm{D}_k[T_m(x)T_n(x)]=\mathrm{D}_k T_m(x)\mathrm{A}_k T_n(x)+\mathrm{A}_k T_m(x)\mathrm{D}_k T_n(x), \qquad m,n\in\mathbb{N}_0.
\end{equation*}
This, however, can easily be seen from the equations \eqref{eq:Dkdef}, \eqref{eq:Akdef}, \eqref{eq:mixedsecond} and \eqref{eq:mixedthird} by the computation
\begin{align*}
&\mathrm{D}_k T_m(x)\mathrm{A}_k T_n(x)+\mathrm{A}_k T_m(x)\mathrm{D}_k T_n(x)\\
&\quad=T_n\left(\left|\cos\left(\frac{\pi}{k}\right)\right|\right)U_{m-1}\left(\left|\cos\left(\frac{\pi}{k}\right)\right|\right)T_n(x)U_{m-1}(x)\\
&\quad\phantom{=}{}+T_m\left(\left|\cos\left(\frac{\pi}{k}\right)\right|\right)U_{n-1}\left(\left|\cos\left(\frac{\pi}{k}\right)\right|\right)T_m(x)U_{n-1}(x)\\
&\quad=\frac{1}{4}\left[U_{m+n-1}\left(\left|\cos\left(\frac{\pi}{k}\right)\right|\right)-U_{n-m-1}\left(\left|\cos\left(\frac{\pi}{k}\right)\right|\right)\right][U_{m+n-1}(x)-U_{n-m-1}(x)]\\
&\quad\phantom{=}{}+\frac{1}{4}\left[U_{m+n-1}\left(\left|\cos\left(\frac{\pi}{k}\right)\right|\right)+U_{n-m-1}\left(\left|\cos\left(\frac{\pi}{k}\right)\right|\right)\right][U_{m+n-1}(x)+U_{n-m-1}(x)]\\
&\quad=\frac{1}{2}U_{m+n-1}\left(\left|\cos\left(\frac{\pi}{k}\right)\right|\right)U_{m+n-1}(x)+\frac{1}{2}U_{n-m-1}\left(\left|\cos\left(\frac{\pi}{k}\right)\right|\right)U_{n-m-1}(x)\\
&\quad=\frac{1}{2}\mathrm{D}_k T_{m+n}(x)+\frac{1}{2}\mathrm{D}_k T_{n-m}(x)=\mathrm{D}_k\left[\frac{1}{2}T_{m+n}(x)+\frac{1}{2}T_{n-m}(x)\right]\\
&\quad=\mathrm{D}_k[T_m(x)T_n(x)]
\end{align*}
for $m\leq n$ (the expansion $T_m(x)T_n(x)=T_{m+n}(x)/2+T_{n-m}(x)/2$ is well known). Now let $n,j\in\mathbb{N}_0$. Via \eqref{eq:recfund}, we compute
\begin{align*}
a_n\mathrm{D}_k P_{n+1}(x)+c_n\mathrm{D}_k P_{n-1}(x)&=\mathrm{D}_k[x P_n(x)]\\
&=\mathrm{D}_k[x]\mathrm{A}_k P_n(x)+\mathrm{A}_k[x]\mathrm{D}_k P_n(x)\\
&=\left|\cos\left(\frac{\pi}{k}\right)\right|x \mathrm{D}_k P_n(x)+\mathrm{A}_k P_n(x);
\end{align*}
multiplication with $P_j(x)$, integration with respect to $\mu$ and the equations \eqref{eq:kappandeffirst} and \eqref{eq:alphandeffirst} yield the second assertion.
\end{proof}

The recurrence relation \eqref{eq:recurrencekappasieved} for $(\kappa_n(\cdot;k))_{n\in\mathbb{N}_0}$ is the analogue to $q$- and non-sieved variants which can be found in \cite{IO11,LO08}.

We now come to the proofs of Theorems~\ref{thm:mainsievedaveraging} and~\ref{thm:mainsievedaskeywilson}.

\begin{proof}[Proof of Theorem~\ref{thm:mainsievedaveraging}]
The case $k=1$ is trivial because $\mathrm{A}_1=\id$, so let $k\geq2$ from now on (in particular, we then have $|\cos(\pi/k)|=\cos(\pi/k)$). ``(i) $\Rightarrow$ (ii)'': let $n\in\mathbb{N}_0$ and $i\in\{0,\ldots,k-1\}$. Using Theorem~\ref{thm:mainsieved} and \eqref{eq:Akdef}, we have
\begin{align*}
&T_{k n+i}\left(\cos\left(\frac{\pi}{k}\right)\right)P_{k n+i}(x;k)-\mathrm{A}_k P_{k n+i}(x;k)\\
&\quad=\sum_{j=0}^{\lfloor\frac{n}{2}\rfloor}p_n(j)\left[T_{k n+i}\left(\cos\left(\frac{\pi}{k}\right)\right)-T_{k n-2j k-i}\left(\cos\left(\frac{\pi}{k}\right)\right)\right]T_{k n-2j k-i}(x)\\
&\quad\phantom{=}{}+\sum_{j=0}^{\lfloor\frac{n}{2}\rfloor}q_n(j)\left[T_{k n+i}\left(\cos\left(\frac{\pi}{k}\right)\right)-T_{k n-2j k+i}\left(\cos\left(\frac{\pi}{k}\right)\right)\right]T_{k n-2j k+i}(x).
\end{align*}
For each $j\in\{0,\ldots,\left\lfloor n/2\right\rfloor\}$ (except the case $n$ even and $j=n/2$, but then $p_n(j)=0$), we compute
\begin{align*}
T_{k n+i}\left(\cos\left(\frac{\pi}{k}\right)\right)-T_{k n-2j k-i}\left(\cos\left(\frac{\pi}{k}\right)\right)&=\cos\left(\frac{(k n+i)\pi}{k}\right)-\cos\left(\frac{(k n-2j k-i)\pi}{k}\right)\\
&=-2\sin((n-j)\pi)\sin\left(\frac{(j k+i)\pi}{k}\right)\\
&=0.
\end{align*}
In the same way, we have
\begin{equation*}
T_{k n+i}\left(\cos\left(\frac{\pi}{k}\right)\right)-T_{k n-2j k+i}\left(\cos\left(\frac{\pi}{k}\right)\right)=0
\end{equation*}
for all $j\in\{0,\ldots,\left\lfloor n/2\right\rfloor\}$. Hence, we see that $P_{k n+i}(x;k)$ is an eigenvector of $\mathrm{A}_k$ which corresponds to the eigenvalue $T_{k n+i}(\cos(\pi/k))$. ``(ii) $\Rightarrow$ (iii)'' is trivial from orthogonality. \text{``(iii) $\Rightarrow$ (i)''}: the condition and Lemma~\ref{lma:last} yield
\begin{equation}\label{eq:centralequation}
4\sum_{j=1}^n a_{j-1}c_j=n+1\quad \text{if}\quad k\,{\not|}\, n\in\mathbb{N}.
\end{equation}
It is immediate from \eqref{eq:centralequation} that $c_1=1/2$. Let $n\in\mathbb{N}$ be such that $k$ does not divide $n+1$. Moreover, assume that, for each $j\in\{1,\ldots,n\}$, $c_j=1/2$ if $j$ is not a multiple of $k$. Decomposing $n+1=k l+i$ with unique $l\in\mathbb{N}_0$, $i\in\{1,\ldots,k-1\}$, on the one hand, we get
\begin{equation*}
k l+i+1=4\sum_{j=1}^{k l+i}a_{j-1}c_j=4\sum_{j=1}^{k l}a_{j-1}c_j+4\sum_{j=k l+1}^{k l+i}a_{j-1}c_j=k l+2c_{k l}+4\sum_{j=k l+1}^{k l+i}a_{j-1}c_j
\end{equation*}
from \eqref{eq:centralequation}, which then simplifies to
\begin{equation}\label{eq:centrallemmafirst}
\frac{c_{k l}}{2}+\sum_{j=k l+1}^{k l+i}a_{j-1}c_j=\frac{i+1}{4}.
\end{equation}
On the other hand, we compute
\begin{equation}\label{eq:centrallemmasecond}
\frac{c_{k l}}{2}+\sum_{j=k l+1}^{k l+i}a_{j-1}c_j=\begin{cases} \dfrac{c_{k l}}{2}+a_{k l}c_{k l+1}, & i=1, \vspace{1mm}\\
 \dfrac{i}{4}+\dfrac{c_{k l+i}}{2}, & \text{else}. \end{cases}
\end{equation}
Combining \eqref{eq:centrallemmafirst} and \eqref{eq:centrallemmasecond}, we obtain that $c_{n+1}=c_{k l+i}=1/2$. This finishes the proof of \text{``(iii) $\Rightarrow$ (i)''}. Concerning the remaining assertion, it suffices to show that \[P_n(\cos(\pi/k);k)=T_n(\cos(\pi/k)),\qquad n\in\mathbb{N}_0;\]
this can be seen as follows: let $l\in\mathbb{N}_0$. Since $T_{k l}(\cos(\pi/k))=(-1)^l$, and since
\begin{equation*}
T_{k l-1}\left(\cos\left(\frac{\pi}{k}\right)\right)=(-1)^l\cos\left(\frac{\pi}{k}\right)=T_{k l+1}\left(\cos\left(\frac{\pi}{k}\right)\right)
\end{equation*}
for $l\neq0$, we have
\begin{equation*}
\cos\left(\frac{\pi}{k}\right)T_{k l}\left(\cos\left(\frac{\pi}{k}\right)\right)=a_l T_{k l+1}\left(\cos\left(\frac{\pi}{k}\right)\right)+c_l T_{k l-1}\left(\cos\left(\frac{\pi}{k}\right)\right).
\end{equation*}
Consequently, the recurrence relations for the sequences \[(P_n(\cos(\pi/k);k))_{n\in\mathbb{N}_0}\qquad \text{and}\qquad (T_n(\cos(\pi/k)))_{n\in\mathbb{N}_0}\] coincide.
\end{proof}

\begin{proof}[Proof of Theorem~\ref{thm:mainsievedaskeywilson}]
First, note that in view of \eqref{eq:kernelstarprior} (iii) is an obvious reformulation of~(ii). Moreover, the case $k=1$ is just Theorem~\ref{thm:lasserobermaierref} because $\mathrm{D}_1=\mathrm{d}/\mathrm{d}x$ (concerning (iv), we refer to the proof given in \cite{LO08}), so let $k\geq2$ from now on. ``(i) $\Rightarrow$ (ii)'': we know from Theorem~\ref{thm:mainsievedaveraging} \text{``(i) $\Rightarrow$ (ii)''} and Lemma~\ref{lma:productrule} that
\begin{align}
&a_n\kappa_{n+1}(j;k)+c_n\kappa_{n-1}(j;k)\nonumber\\
&\quad=\cos\left(\frac{\pi}{k}\right)[a_j\kappa_n(j+1;k)+c_j\kappa_n(j-1;k)]+T_n\left(\cos\left(\frac{\pi}{k}\right)\right)\frac{\delta_{n,j}}{h(n)}, \qquad n,j\in\mathbb{N}_0,\label{eq:productrulespecial}
\end{align}
and we now use induction on $n$ to deduce that
\begin{equation}\label{eq:kappaind}
\kappa_n(n-1-2j;k)=\frac{U_{n-1}\left(\cos\left(\frac{\pi}{k}\right)\right)}{c_n h(n)}, \qquad j\in\left\{0,\ldots,\left\lfloor\frac{n-1}{2}\right\rfloor\right\}
\end{equation}
for each $n\in\mathbb{N}$, which implies the assertion due to \eqref{eq:kernelstarprior} and \eqref{eq:kappandefsecond}. It is obvious from Lemma~\ref{lma:last} that \eqref{eq:kappaind} is satisfied for $n\in\{1,2\}$, so let $n\in\mathbb{N}\backslash\{1\}$ be arbitrary but fixed now and assume~\eqref{eq:kappaind} to hold for both $n-1$ and $n$. If $j\in\{1,\ldots,\left\lfloor n/2\right\rfloor\}$, then \eqref{eq:productrulespecial} implies
\begin{equation}\label{eq:threecases}
c_{n+1}h(n+1)\kappa_{n+1}(n-2j;k)=\cos\left(\frac{\pi}{k}\right)\frac{U_{n-1}\left(\cos\left(\frac{\pi}{k}\right)\right)}{c_n}-U_{n-2}\left(\cos\left(\frac{\pi}{k}\right)\right)\frac{a_{n-1}}{c_{n-1}},
\end{equation}
and we distinguish two cases: if $k\,{\not|}\,n$, then \eqref{eq:threecases} yields
\begin{equation*}
c_{n+1}h(n+1)\kappa_{n+1}(n-2j;k)=2\cos\left(\frac{\pi}{k}\right)U_{n-1}\left(\cos\left(\frac{\pi}{k}\right)\right)-U_{n-2}\left(\cos\left(\frac{\pi}{k}\right)\right)
\end{equation*}
because $a_{n-1}=c_{n-1}=1/2$ (if $k\,{\not|}\,(n-1)$) or $U_{n-2}(\cos(\pi/k))=0$ (if $k|(n-1)$), which, by the recurrence relation for the Chebyshev polynomials of the second kind, simplifies to
\begin{equation*}
c_{n+1}h(n+1)\kappa_{n+1}(n-2j;k)=U_n\left(\cos\left(\frac{\pi}{k}\right)\right).
\end{equation*}
If $k|n$, however, we compute from \eqref{eq:threecases}
\begin{gather*}
c_{n+1}h(n+1)\kappa_{n+1}(n-2j;k)=-U_{n-2}\left(\cos\left(\frac{\pi}{k}\right)\right)=U_n\left(\cos\left(\frac{\pi}{k}\right)\right).
\end{gather*}
If one takes into account Lemma~\ref{lma:last} for the remaining case ``$j=0$'' again, the proof of the direction ``(i) $\Rightarrow$ (ii)'' is finished. ``(ii) $\Rightarrow$ (iv)'' is trivial from \eqref{eq:kernelstarprior}. The direction \text{``(iv) $\Rightarrow$ (i)''} is more involved; we first deal with $c_1$ and $c_2$: due to the conditions $\kappa_3(0;k)=\sigma(3;k)$ and $\kappa_4(1;k)=\sigma(4;k)$, a tedious but straightforward calculation based on Lemmas~\ref{lma:last} and~\ref{lma:productrule} yields
\begin{gather}\label{eq:initialfirst}
U_2\left(\cos\left(\frac{\pi}{k}\right)\right)[1-4c_1+4c_1c_2]=3-4c_1-4c_2+4c_1c_2
\end{gather}
and
\begin{align*}
&\cos\left(\frac{\pi}{k}\right)U_2\left(\cos\left(\frac{\pi}{k}\right)\right)[3-4c_1-8c_2+4c_1c_2+8c_2c_3]\\
&\quad=\cos\left(\frac{\pi}{k}\right)[9-12c_1-12c_2-8c_3+12c_1c_2+16c_2c_3].
\end{align*}
Combining these equations, we obtain
\begin{equation}\label{eq:initialfourth}
\cos\left(\frac{\pi}{k}\right)U_2\left(\cos\left(\frac{\pi}{k}\right)\right)[c_1-c_2-c_1c_2+c_2c_3]=\cos\left(\frac{\pi}{k}\right)[-c_3+2c_2c_3].
\end{equation}
To get another equation for $c_1$, $c_2$, $c_3$, we first note that Lemma~\ref{lma:last} and \eqref{eq:Chebrecurrencesim} imply that
\begin{align*}
\mathrm{A}_k P_4(x)={}&T_4\left(\cos\left(\frac{\pi}{k}\right)\right)P_4(x)\\
&+\left[T_2\left(\cos\left(\frac{\pi}{k}\right)\right)-T_4\left(\cos\left(\frac{\pi}{k}\right)\right)\right]\frac{a_1-c_3}{a_3}P_2(x)+\alpha_4(0;k)\\
={}&T_4\left(\cos\left(\frac{\pi}{k}\right)\right)[P_4(x)-1]\\&+\left[T_2\left(\cos\left(\frac{\pi}{k}\right)\right)-T_4\left(\cos\left(\frac{\pi}{k}\right)\right)\right]\frac{a_1-c_3}{a_3}[P_2(x)-1]+\mathrm{A}_k P_4(1),
\end{align*}
which, since
\begin{equation*}
P_4(x)=\frac{1}{8a_1a_2a_3}T_4(x)+\frac{a_1-c_3}{2a_1a_3}T_2(x)+\frac{8a_1a_2a_3-4a_1a_2+4a_2c_3-1}{8a_1a_2a_3}
\end{equation*}
due to \eqref{eq:recfund}, becomes
\begin{align}
\mathrm{A}_k P_4(x)={}&T_4\left(\cos\left(\frac{\pi}{k}\right)\right)P_4(x)+\left[T_2\left(\cos\left(\frac{\pi}{k}\right)\right) -T_4\left(\cos\left(\frac{\pi}{k}\right)\right)\right]\frac{a_1-c_3}{a_3}P_2(x)\nonumber\\
&+\left[T_4\left(\cos\left(\frac{\pi}{k}\right)\right)-1\right]\left[\frac{1}{8a_1a_2a_3}-1\right]-\left[1-T_2\left(\cos\left(\frac{\pi}{k}\right)\right)\right]
\frac{a_1-c_3}{2a_1a_3}\nonumber\\
&-\left[T_2\left(\cos\left(\frac{\pi}{k}\right)\right)-T_4\left(\cos\left(\frac{\pi}{k}\right)\right)\right]\frac{a_1-c_3}{a_3}\label{eq:initialthirdpre}
\end{align}
after another tedious but straightforward calculation which uses \eqref{eq:Akdef}. We then take into account that Lemma~\ref{lma:productrule} and the conditions $\kappa_3(0;k)=\sigma(3;k)$, $\kappa_4(1;k)=\sigma(4;k)$ and $\kappa_5(0;k)=\kappa_5(2;k)(=\sigma(5;k))$ yield $\alpha_4(0;k)=\alpha_4(2;k)$, which can be rewritten as
\begin{align}
&U_2\left(\cos\left(\frac{\pi}{k}\right)\right)\big[1-8c_1+8c_1^2+16c_1c_2-16c_1^2c_2-8c_1c_2^2-8c_1c_2c_3+8c_1^2c_2^2+8c_1c_2^2c_3\big]\nonumber\\
&\quad=3-4c_1-4c_2-4c_3+4c_1c_2+8c_1c_3+4c_2c_3-8c_1c_2c_3\label{eq:initialthird}
\end{align}
as a consequence of \eqref{eq:initialthirdpre} and the simplifications
\begin{align*}
&T_2\left(\cos\left(\frac{\pi}{k}\right)\right)-T_4\left(\cos\left(\frac{\pi}{k}\right)\right)=2U_2\left(\cos\left(\frac{\pi}{k}\right)\right)\sin^2\left(\frac{\pi}{k}\right),\\
&T_4\left(\cos\left(\frac{\pi}{k}\right)\right)-1=-\left[2U_2\left(\cos\left(\frac{\pi}{k}\right)\right)+2\right]\sin^2\left(\frac{\pi}{k}\right),\\
&1-T_2\left(\cos\left(\frac{\pi}{k}\right)\right)=2\sin^2\left(\frac{\pi}{k}\right).
\end{align*}
Now, the combination of \eqref{eq:initialfirst} with \eqref{eq:initialthird} gives
\begin{align*}
&U_2\left(\cos\left(\frac{\pi}{k}\right)\right)\big[c_1-2c_1^2-3c_1c_2+4c_1^2c_2+2c_1c_2^2+2c_1c_2c_3-2c_1^2c_2^2-2c_1c_2^2c_3\big]\\
&\quad=c_3-2c_1c_3-c_2c_3+2c_1c_2c_3;
\end{align*}
dividing this by $1-c_2$, we finally obtain
\begin{equation}\label{eq:initialfifth}
U_2\left(\cos\left(\frac{\pi}{k}\right)\right)\big[c_1-2c_1^2-2c_1c_2+2c_1^2c_2+2c_1c_2c_3\big]=c_3-2c_1c_3.
\end{equation}
From now on, we consider the nonlinear system \eqref{eq:initialfirst}, \eqref{eq:initialfourth}, \eqref{eq:initialfifth}, which depends on $k$. If $k=2$, then \eqref{eq:initialfirst} immediately gives $c_1=1/2$ because $U_2(0)=-1$. In the following, let $k\geq3$; here, $\cos(\pi/k)\neq0$ and \eqref{eq:initialfourth} reduces to
\begin{equation}\label{eq:initialfourthred}
U_2\left(\cos\left(\frac{\pi}{k}\right)\right)[c_1-c_2-c_1c_2+c_2c_3]=-c_3+2c_2c_3.
\end{equation}
Combining \eqref{eq:initialfirst} and \eqref{eq:initialfifth}, we see that
\begin{align}
&[3-4c_1-4c_2+4c_1c_2]\big[c_1-2c_1^2-2c_1c_2+2c_1^2c_2+2c_1c_2c_3\big]\nonumber\\
&\quad=U_2\left(\cos\left(\frac{\pi}{k}\right)\right)[1-4c_1+4c_1c_2]\big[c_1-2c_1^2-2c_1c_2+2c_1^2c_2+2c_1c_2c_3\big]\nonumber\\
&\quad=[1-4c_1+4c_1c_2][c_3-2c_1c_3].\label{eq:initialI}
\end{align}
In the same way combining \eqref{eq:initialfirst}, \eqref{eq:initialfourthred} on the one hand and \eqref{eq:initialfifth}, \eqref{eq:initialfourthred} on the other hand, we get
\begin{align*}
&[3-4c_1-4c_2+4c_1c_2][c_1-c_2-c_1c_2+c_2c_3]=[1-4c_1+4c_1c_2][-c_3+2c_2c_3],\\
&[c_3-2c_1c_3][c_1-c_2-c_1c_2+c_2c_3]=\big[c_1-2c_1^2-2c_1c_2+2c_1^2c_2+2c_1c_2c_3\big][-c_3+2c_2c_3],
\end{align*}
which can be rewritten as
\begin{align}
&\big[1-4c_1+c_2+8c_1c_2-4c_2^2-4c_1c_2^2\big]c_3\nonumber\\
&\quad=-3c_1+3c_2+4c_1^2+3c_1c_2-4c_2^2-8c_1^2c_2+4c_1^2c_2^2,\label{eq:initialIIadd} \\
&\big[c_2-4c_1c_2^2\big]c_3=-2c_1+c_2+4c_1^2+3c_1c_2-8c_1^2c_2-4c_1c_2^2+4c_1^2c_2^2.\label{eq:initialII}
\end{align}
Moreover, we can deduce from \eqref{eq:initialIIadd} and \eqref{eq:initialII} that
\begin{align*}
&\big[-3c_1+3c_2+4c_1^2+3c_1c_2-4c_2^2-8c_1^2c_2+4c_1^2c_2^2\big]\big[c_2-4c_1c_2^2\big]\\
&\quad=\big[1-4c_1+c_2+8c_1c_2-4c_2^2-4c_1c_2^2\big]c_3\big[c_2-4c_1c_2^2\big]\\
&\quad=\big[1-4c_1+c_2+8c_1c_2-4c_2^2-4c_1c_2^2\big]\\
&\phantom{\quad=}{}\times\big[-2c_1+c_2+4c_1^2+3c_1c_2-8c_1^2c_2-4c_1c_2^2+4c_1^2c_2^2\big],
\end{align*}
which now considerably simplifies to
\begin{equation*}
[4c_1c_2-4c_1+1][2c_1c_2-2c_1+c_2][2c_2-1][2c_1-1]=0.
\end{equation*}
If $4c_1c_2-4c_1+1=0$, it is immediate from \eqref{eq:initialfirst} that $c_2=1/2$. If $2c_1c_2-2c_1+c_2=0$, then~\eqref{eq:initialfirst} reads $U_2(\cos(\pi/k))[1-2c_2]=3-6c_2$, and we can conclude that $c_2=1/2$, too \big(because $U_2(\cos(\pi/k))=4(\cos(\pi/k))^2-1<3$\big). If $c_2=1/2$, then \eqref{eq:initialII} reduces to $(1-2c_1)(1-c_1-c_3)=0$, and if additionally $c_3=1-c_1$, then \eqref{eq:initialI} implies that $c_1=1/2$. Finally, if $c_1=1/2$, then~${c_2=1/2}$ because otherwise \eqref{eq:initialfirst} would yield ${-1\!=\!U_2(\cos(\pi/k))\!=\!\sin(3\pi/k)/\sin(\pi/k)\!\geq\!0}$. Therefore, allowing $k\geq2$ to be arbitrary again, we have seen that $c_1=1/2$ in each possible case, and that $c_2=1/2$ if $k\neq2$. Now let $n\in\mathbb{N}\backslash\{1\}$ be arbitrary but fixed and assume that, for all $j\in\{1,\ldots,n\}$, $c_j=1/2$ if $j$ is not a multiple of $k$. Lemma~\ref{lma:productrule} and the conditions $\kappa_{n+1}(n-2;k)=\sigma(n+1;k)$ and $\kappa_{n+2}(n-1;k)=\sigma(n+2;k)$ yield
\begin{equation}\label{eq:indsaw1}
a_{n+1}\sigma(n+2;k)+c_{n+1}\sigma(n;k)=\cos\left(\frac{\pi}{k}\right)\sigma(n+1;k)+\alpha_{n+1}(n-1;k).
\end{equation}
Since $\alpha_{n+1}(n-1;k)$ is uniquely determined by the recurrence coefficients $c_1,\ldots,c_n$ (which is already a consequence of \eqref{eq:alphandefsecond}), we have $\alpha_{n+1}(n-1;k)=0$ by Theorem~\ref{thm:mainsievedaveraging} ``(i) $\Rightarrow$ (ii)'' and the induction hypothesis.\footnote{Such arguments will occur several times in our proof, so we give some more details at this stage: compare $(P_j(x))_{j\in\mathbb{N}_0}$ to the sequence \smash{$\big(\widetilde{P}_j(x)\big)_{j\in\mathbb{N}_0}$} which is given by $\widetilde{c}_j:=c_j$ for $j\in\{1,\ldots,n\}$ and $\widetilde{c}_j:=1/2$ otherwise. Then \smash{$\big(\widetilde{P}_j(x)\big)_{j\in\mathbb{N}_0}$} is $k$-sieved, so $0=\widetilde{\alpha}_{n+1}(n-1;k)=\alpha_{n+1}(n-1;k)$. We used similar ideas also in our paper~\cite{Ka16a}.} Therefore, \eqref{eq:indsaw1} simplifies to
\begin{equation*}
a_{n+1}\sigma(n+2;k)+c_{n+1}\sigma(n;k)=\cos\left(\frac{\pi}{k}\right)\sigma(n+1;k),
\end{equation*}
and then multiplication with $a_n c_n h(n)$, Lemma~\ref{lma:last} and the recurrence relation for the Chebyshev polynomials of the second kind imply that
\begin{equation}\label{eq:indsaw2}
\cos\left(\frac{\pi}{k}\right)U_n\left(\cos\left(\frac{\pi}{k}\right)\right)(1-2c_{n+1})c_n=U_{n-1}\left(\cos\left(\frac{\pi}{k}\right)\right)(1-2c_n)c_{n+1}.
\end{equation}
If $k\geq3$ and if $n+1$ is not a multiple of $k$, then $\cos(\pi/k)U_n(\cos(\pi/k))\neq0$; consequently, \eqref{eq:indsaw2} tells $c_{n+1}=1/2$ because $U_{n-1}(\cos(\pi/k))=0$ or $c_n=1/2$, depending on whether $k$ divides $n$ or not. The case $k=2$ is more involved---here, \eqref{eq:indsaw2} does not allow for any conclusion. Instead, we apply an idea which we similarly used in \cite{Ka16a} (at that time for the operators $\mathcal{D}_q$ and $\mathcal{A}_q$; cf. particularly \cite[Lemma 3.1, proof of Theorem 2.3]{Ka16a}) and proceed as follows: by the conditions, there is some $m\in\{0,\ldots,\left\lfloor(n-2)/2\right\rfloor\}$ such that ${\kappa_{n+3}(n-2-2m;2)=\sigma(n+3;2)}$. Since $c_1,\ldots,c_n$ fix the first $n+2$ polynomials $P_0(x),\ldots,P_{n+1}(x)$, $c_1,\ldots,c_n$ yield unique $\lambda_0,\ldots,\lambda_{n+2},\nu_0,\ldots,\nu_{n-1}\in\mathbb{R}$ such that
\begin{align*}
\mathrm{A}_2[x P_{n+1}(x)]={}&\lambda_0+x\sum_{j=0}^{n+1}\lambda_{j+1}P_j(x)\\
={}&\lambda_0+\lambda_2c_1+\sum_{j=1}^n[\lambda_j a_{j-1}+\lambda_{j+2}c_{j+1}]P_j(x)\\
&+\lambda_{n+1}a_n P_{n+1}(x)+\lambda_{n+2}a_{n+1} P_{n+2}(x)\\
={}&\sum_{j=0}^{n-1}\nu_j P_j(x)+[\lambda_n a_{n-1}+\lambda_{n+2}c_{n+1}]P_n(x)\\
&+\lambda_{n+1}a_n P_{n+1}(x)+\lambda_{n+2}a_{n+1} P_{n+2}(x).
\end{align*}
Therefore, $c_1,\ldots,c_n$ determine the integral
\begin{equation*}
\int_\mathbb{R}\!\mathrm{A}_2[x P_{n+1}(x)]P_{n-2-2m}(x)\,\mathrm{d}\mu(x)=\frac{\nu_{n-2-2m}}{h(n-2-2m)}
\end{equation*}
uniquely, so \eqref{eq:recfund} and Theorem~\ref{thm:mainsievedaveraging} ``(i) $\Rightarrow$ (ii)'' imply that
\begin{equation*}
\int_\mathbb{R}\!\mathrm{A}_2[x P_{n+1}(x)]P_{n-2-2m}(x)\,\mathrm{d}\mu(x)=0
\end{equation*}
as a consequence of the induction hypothesis. Hence, using Theorem~\ref{thm:mainsievedaveraging} ``(i) $\Rightarrow$ (ii)'' again, \eqref{eq:recfund} and the fact that $\alpha_n(n-2-2m;2)$ is uniquely determined by $c_1,\ldots,c_n$, too (use \eqref{eq:alphandefsecond}), we see that
\begin{equation}\label{eq:alphaequalI}
0=a_{n+1}\alpha_{n+2}(n-2-2m;2)+c_{n+1}\alpha_n(n-2-2m;2)=a_{n+1}\alpha_{n+2}(n-2-2m;2).
\end{equation}
Furthermore, we have
\begin{equation}\label{eq:alphaequalII}
\alpha_{n+2}(n-2-2m;2)=\alpha_{n+2}(n;2),
\end{equation}
which can be seen as follows: on the one hand, since $\kappa_{n+1}(n-2-2m;2)$ is uniquely determined by $c_1,\ldots,c_n$ (this is a consequence of \eqref{eq:kappandefsecond}), one has $\kappa_{n+1}(n-2-2m;2)=\sigma(n+1;2)$ by the already established direction ``(i) $\Rightarrow$ (ii)''; hence, by Lemma~\ref{lma:productrule} we have
\begin{equation*}
a_{n+2}\sigma(n+3;2)+c_{n+2}\sigma(n+1;2)=\alpha_{n+2}(n-2-2m;2).
\end{equation*}
On the other hand, by Lemma~\ref{lma:productrule} and the condition $\kappa_{n+3}(n;2)=\sigma(n+3;2)$, we have
\begin{equation*}
a_{n+2}\sigma(n+3;2)+c_{n+2}\sigma(n+1;2)=\alpha_{n+2}(n;2),
\end{equation*}
so \eqref{eq:alphaequalII} is established. Combining \eqref{eq:alphaequalII} with \eqref{eq:alphaequalI}, we obtain that
\begin{equation*}
\alpha_{n+2}(n;2)=0.
\end{equation*}
Now following the proof of Theorem~\ref{thm:mainsievedaveraging} ``(iii) $\Rightarrow$ (i)'', we can conclude that $c_{n+1}=1/2$ if $n+1$ is odd, which finishes the proof of Theorem~\ref{thm:mainsievedaskeywilson}.
\end{proof}

\begin{Remark}
We note that the implication ``(iv) $\Rightarrow$ (i)'' of Theorem~\ref{thm:mainsievedaskeywilson} remains valid if (iv) is weakened in the following way:
\begin{itemize}\itemsep=0pt
\item If $k\in\{1,2\}$, it suffices to require that $\kappa_{2n+1}(2n-2;k)=\sigma(2n+1;k)$ for all $n\in\mathbb{N}$ and that for every $n\in\mathbb{N}$ there is an $m\in\{0,\ldots,n-1\}$ such that $\kappa_{2n+3}(2m;k)=\sigma(2n+3;k)$. For $k=1$ (which yields a characterization of the ultraspherical polynomials), this is a~consequence of \cite[Theorem~2.1]{Ka16a}. The case $k=2$, however, is obvious from Theorem~\ref{thm:mainsievedaskeywilson} because $\mathrm{D}_2 P_{2n}(x)=0$ $(n\in\mathbb{N}_0)$ by \eqref{eq:Dkdef} without any further restriction.
\item If $k\geq3$, one can drop the requirement ``there is an $m\in\{0,\ldots,\left\lfloor(n-1)/2\right\rfloor\}$ such that $\kappa_{n+4}(n-1-2m;k)=\sigma(n+4;k)$'' for all $n\geq2$. This can be seen from the proof above.
\end{itemize}
\end{Remark}

\subsection*{Acknowledgements}

The research was begun when the author worked at Technical University of Munich, and the author gratefully acknowledges support from the graduate program TopMath of the ENB (Elite Network of Bavaria) and the TopMath Graduate Center of TUM Graduate School at Technical University of Munich. The research was continued and completed at RWTH Aachen University. The author also thanks the referees for carefully reading the manuscript, as well as for their valuable comments.

\pdfbookmark[1]{References}{ref}
\LastPageEnding


\begin{thebibliography}{99}
\footnotesize\itemsep=0pt

\bibitem{AS65}
Abramowitz M., Stegun I.A., Handbook of mathematical functions with formulas,
  graphs, and mathematical tables, \textit{National Bureau of Standards Applied
  Mathematics Series}, Vol.~55, U.S.~Government Printing Office, Washington,
  DC, 1964.

\bibitem{Al90}
Al-Salam W., Characterization theorems for orthogonal polynomials, in
  Orthogonal {P}olynomials, \textit{NATO Adv. Sci. Inst. Ser.~C: Math. Phys.
  Sci.}, Vol. 294, \href{https://doi.org/10.1007/978-94-009-0501-6_1}{Kluwer Academic Publishers Group}, Dordrecht, 1990, 1--24.

\bibitem{AAA84}
Al-Salam W., Allaway W.R., Askey R., Sieved ultraspherical polynomials,
  \href{https://doi.org/10.2307/1999273}{\textit{Trans. Amer. Math. Soc.}} \textbf{284} (1984), 39--55.

\bibitem{CdJP20}
Castillo K., de~Jesus M.N., Petronilho J., An electrostatic interpretation of
  the zeros of sieved ultraspherical polynomials, \href{https://doi.org/10.1063/1.5063333}{\textit{J.~Math. Phys.}}
  \textbf{61} (2020), 053501, 19~pages, \href{https://arxiv.org/abs/1909.12062}{arXiv:1909.12062}.

\bibitem{Ch78}
Chihara T.S., An introduction to orthogonal polynomials, \textit{Math. Appl.},
  Vol.~13, Gordon and Breach, New York, 1978.

\bibitem{CSvD98}
Coolen-Schrijner P., van Doorn E.A., Analysis of random walks using orthogonal
  polynomials,  \href{https://doi.org/10.1016/S0377-0427(98)00172-1}{\textit{J.~Comput. Appl. Math.}} \textbf{99} (1998), 387--399.

\bibitem{GvA88}
Geronimo J.S., Van~Assche W., Orthogonal polynomials on several intervals via a
  polynomial mapping, \href{https://doi.org/10.2307/2001092}{\textit{Trans. Amer. Math. Soc.}} \textbf{308} (1988),
  559--581.

\bibitem{Is09}
Ismail M.E.H., Classical and quantum orthogonal polynomials in one variable,
  \textit{Encyclopedia of Mathematics and its Applications}, Vol.~98,  \href{https://doi.org/10.1017/CBO9781107325982}{Cambridge
  University Press}, Cambridge, 2009.

\bibitem{IL92}
Ismail M.E.H., Li X., On sieved orthogonal polynomials.~{IX}. {O}rthogonality
  on the unit circle, \href{https://doi.org/10.2140/pjm.1992.153.289}{\textit{Pacific J.~Math.}} \textbf{153} (1992), 289--297.

\bibitem{IO11}
Ismail M.E.H., Obermaier J., Characterizations of continuous and discrete
  {$q$}-ultraspherical polynomials, \href{https://doi.org/10.4153/CJM-2010-080-0}{\textit{Canad.~J. Math.}} \textbf{63}
  (2011), 181--199.

\bibitem{IP12}
Ismail M.E.H., Simeonov P., Connection relations and characterizations of
  orthogonal polynomials,  \href{https://doi.org/10.1016/j.aam.2012.04.004}{\textit{Adv. in Appl. Math.}} \textbf{49} (2012),
  134--164.

\bibitem{Ka16b}
Kahler S., Characterizations of orthogonal polynomials and harmonic analysis on
  polynomial hypergroups, {D}issertation, {T}echnical {U}niversity of {M}unich, 2016,
  available at
  \url{https://nbn-resolving.de/urn/resolver.pl?urn:nbn:de:bvb:91-diss-20160530-1289608-1-3}.

\bibitem{Ka16a}
Kahler S., Characterizations of ultraspherical polynomials and their
  {$q$}-analogues, \href{https://doi.org/10.1090/proc/12640}{\textit{Proc. Amer. Math. Soc.}} \textbf{144} (2016),
  87--101.

\bibitem{KLS10}
Koekoek R., Lesky P.A., Swarttouw R.F., Hypergeometric orthogonal polynomials
  and their {$q$}-analogues, \textit{Springer Monogr. Math.}, \href{https://doi.org/10.1007/978-3-642-05014-5}{Springer}, Berlin, 2010.

\bibitem{LO08}
Lasser R., Obermaier J., A new characterization of ultraspherical polynomials,
  \href{https://doi.org/10.1090/S0002-9939-08-09378-7}{\textit{Proc. Amer. Math. Soc.}} \textbf{136} (2008), 2493--2498.

\bibitem{vDS93}
van Doorn E.A., Schrijner P., Random walk polynomials and random walk measures,
  \href{https://doi.org/10.1016/0377-0427(93)90162-5}{\textit{J.~Comput. Appl. Math.}} \textbf{49} (1993), 289--296.

\bibitem{WLXZ16}
Wu X.-B., Lin Y., Xu S.-X., Zhao Y.-Q., Plancherel--{R}otach type asymptotics of
  the sieved {P}ollaczek polynomials via the {R}iemann--{H}ilbert approach,
  \href{https://doi.org/10.1016/j.jat.2016.04.002}{\textit{J.~Approx. Theory}} \textbf{208} (2016), 21--58.

\end{thebibliography}
\end{document}